\documentclass[oneside,final,11pt]{amsart}

\usepackage{dsfont}
\usepackage[all]{xy}

\newcommand{\R}{\mathds R}
\newcommand{\chilow}[1]{\chi_{\lower2pt\hbox{$\scriptstyle#1$}}}

\newcommand{\Clopc}[1]{\mathop{\mathrm{Clop}_{\mathrm c}^{#1}}}

\DeclareMathOperator{\LEP}{LEP}
\DeclareMathOperator{\Fin}{Fin}
\DeclareMathOperator{\Int}{int}
\DeclareMathOperator{\atom}{atom}

\DeclareMathOperator{\rc}{rc}
\DeclareMathOperator{\cb}{cb}

\DeclareMathOperator{\clop}{Clop}
\DeclareMathOperator{\clopc}{Clop_c}
\DeclareMathOperator{\MA}{MA}

\title{Local extension property for finite height spaces}

\author{Claudia Correa}
\address{Centro de Matem\'atica, Computa\c c\~ao e Cogni\c c\~ao,\hfill\break\indent Universidade Federal do ABC, Brazil}
\email{claudiac.mat@gmail.com} \urladdr{http://professor.ufabc.edu.br/\~{}claudia.correa}
\author{Daniel V. Tausk}
\address{Departamento de Matem\'atica,\hfill\break\indent Universidade de S\~ao Paulo, Brazil}
\email{tausk@ime.usp.br} \urladdr{http://www.ime.usp.br/\~{}tausk}

\date{June 11th, 2018}

\begin{document}

\theoremstyle{plain}\newtheorem{teo}{Theorem}[section]
\theoremstyle{plain}\newtheorem{prop}[teo]{Proposition}
\theoremstyle{plain}\newtheorem{lem}[teo]{Lemma}
\theoremstyle{plain}\newtheorem{cor}[teo]{Corollary}
\theoremstyle{definition}\newtheorem{defin}[teo]{Definition}
\theoremstyle{remark}\newtheorem{rem}[teo]{Remark}
\theoremstyle{plain} \newtheorem{assum}[teo]{Assumption}
\theoremstyle{definition}\newtheorem{example}[teo]{Example}
\theoremstyle{plain}\newtheorem*{conjecture}{Conjecture}

\begin{abstract}
We introduce a new technique for the study of the local extension property ($\LEP$) for boolean algebras and we use it to show that the clopen algebra of every compact Hausdorff space $K$ of finite height has $\LEP$. This implies, under appropriate additional assumptions on $K$ and Martin's Axiom, that every
twisted sum of $c_0$ and $C(K)$ is trivial, generalizing a recent result by Marciszewski and Plebanek.
\end{abstract}

\maketitle

\begin{section}{Introduction}

The purpose of this article is threefold: first (Section~\ref{sec:main}), we present a generalization of the construction of $\Psi$-spaces (\cite{Hrusak,Mrowka}) that is relevant for understanding the clopen algebra of a compact Hausdorff scattered space. The second purpose (Section~\ref{sec:pointdiagram}) is to introduce a combinatorial method for the study of embeddings of pairs of finite boolean algebras, which is useful to establish that a given boolean algebra has the local extension property ($\LEP$). Finally (Section~\ref{sec:mainresults}), we use both
techniques to prove that the clopen algebra of every compact Hausdorff space of finite height has $\LEP$.

The local extension property for boolean algebras (see Definition~\ref{definLep} and Remark~\ref{remLEP}) was introduced by Marciszewski and Plebanek in \cite{Plebanek} with the purpose of proving that,
under Martin's Axiom ($\MA$), it holds that every twisted sum of $c_0$ and $C(K)$ is trivial, for certain nonmetrizable compact Hausdorff spaces $K$.
The examples presented in \cite{Plebanek} were the first (consistent) counterexamples to the conjecture of Cabello, Castillo, Kalton, and Yost
\cite{CastilloKalton,JesusSobczyk,Castilloscattered,ExtCKc0} that, for any nonmetrizable compact Hausdorff space $K$, there exists a {\em nontrivial twisted sum of $c_0$ and $C(K)$}, i.e., a Banach space $X$ containing a noncomplemented isomorphic copy
$Y$ of $c_0$ such that $X/Y$ is isomorphic to $C(K)$. As usual, $C(K)$ denotes the Banach space of continuous real-valued functions on $K$, endowed with
the supremum norm.

In \cite[Theorem~3.4, Theorem~5.1]{Plebanek}, the authors show that, under Martin's Axiom, given a separable zero-dimensional compact Hausdorff space $K$ with weight less than the {\it continuum}, if the algebra $\clop(K)$ of clopen subsets of $K$ has $\LEP$, then every twisted sum of $c_0$ and $C(K)$ is trivial. They also show that a free boolean algebra (equivalently, the clopen algebra of a power of $2$) has $\LEP$ (\cite[Proposition~4.7]{Plebanek}) and that if $K$ is the one-point compactification of a $\Psi$-space, then $\clop(K)$ has $\LEP$ (\cite[Proposition~4.9]{Plebanek}).

Recall that a {\em $\Psi$-space\/} is a separable locally compact Hausdorff space of height $2$. The standard construction of a $\Psi$-space is
done using an almost disjoint family $(V_\gamma)_{\gamma\in\Gamma}$ of subsets of $\omega$ which defines a topology in the disjoint union
$\omega\cup\Gamma$. In this topology, $\omega$ is an open dense discrete subspace and the fundamental neighborhoods of each $\gamma\in\Gamma$ are unions
of $\{\gamma\}$ with cofinite subsets of $V_\gamma$. In Section~\ref{sec:main}, we generalize this construction by replacing the countable discrete space
$\omega$ with an arbitrary locally compact Hausdorff space and the almost disjoint family $(V_\gamma)_{\gamma\in\Gamma}$
with an antichain in an appropriate boolean algebra. This generalization allows us to iterate the process in order to produce an arbitrary compact Hausdorff
space of finite height, obtaining then a description of its clopen algebra in terms of the various antichains used in the construction.
Though this iterative process can be used a transfinite number of times to obtain a similar description for the clopen algebra of an arbitrary
compact Hausdorff scattered space, we leave open the problem of weather the clopen algebra has $\LEP$ for spaces of infinite height.

\end{section}

\begin{section}{Generalized $\Psi$-spaces}
\label{sec:main}

Throughout this section, $\mathcal X$ denotes a topological space and $\Gamma$ denotes an arbitrary set with $\mathcal X \cap \Gamma=\emptyset$. We denote by $\wp(\mathcal X)$ the boolean algebra of all subsets of $\mathcal X$, by $\cb(\mathcal X)$ the boolean subalgebra of $\wp(\mathcal X)$ consisting of sets with compact boundary and by $\rc(\mathcal X)$ the ideal of $\cb(\mathcal X)$ consisting of relatively compact sets, i.e., sets with compact closure. Our initial purpose is to characterize the topologies on $\mathcal X \cup \Gamma$ for which $\mathcal X$ is an open dense subspace and $\Gamma$ is a discrete subspace. This characterization is given in terms of families of filters.

\begin{defin}\label{filtertoplogy}
By the topology on $\mathcal X \cup \Gamma$ {\em induced\/} by a given a family $(\mathfrak F_\gamma)_{\gamma \in \Gamma}$ of proper filters of $\wp(\mathcal X)$, we mean the topology in which a subset $U$ of $\mathcal X \cup \Gamma$ is open if and only if $U \cap \mathcal X$ is open in $\mathcal X$ and $U \cap \mathcal X$ is in $\mathfrak F_\gamma$, for all $\gamma \in U \cap \Gamma$.
\end{defin}

In order to obtain an one-to-one correspondence between certain topologies on $\mathcal X\cup\Gamma$ and families of filters, we need to restrict ourselves
to filters satisfying an additional condition.
\begin{defin}
We say that a filter $\mathfrak F$ of $\wp(\mathcal X)$ is {\em closed under interiors\/} if for every $V\in\mathfrak F$, the interior
$\Int(V)$ of $V$ belongs to $\mathfrak F$.
\end{defin}

\begin{prop}\label{TopFiltro}
The topology on $\mathcal X \cup \Gamma$ induced by a family of proper filters $(\mathfrak F_\gamma)_{\gamma \in \Gamma}$ satisfies the following conditions:
\begin{itemize}
\item[(a)] the subspace topology of $\mathcal X$ coincides with the original topology of $\mathcal X$;
\item[(b)] $\mathcal X$ is open and dense in $\mathcal X \cup \Gamma$;
\item[(c)] the subspace topology of $\Gamma$ is discrete.
\end{itemize}
If each filter $\mathfrak F_\gamma$ is closed under interiors, then:
\begin{gather}
\label{Fgamma}\mathfrak F_\gamma=\big\{U \cap \mathcal X: \text{$U$ is a nhood of $\gamma$ in $\mathcal X \cup \Gamma$}\big\} \quad \text{and}\\
\label{Fgamma2}\mathfrak F_\gamma=\big\{V \in \wp(\mathcal X): \text{$V \cup \{\gamma\}$ is a nhood of $\gamma$ in $\mathcal X \cup \Gamma$}\big\}.
\end{gather}
Moreover, given a topology on $\mathcal X \cup \Gamma$ satisfying (a), (b) and (c), the sets defined in \eqref{Fgamma} are proper closed under interiors filters of $\wp(\mathcal X)$ and the given topology on $\mathcal X \cup \Gamma$ coincides with the topology induced by this family of filters.
\end{prop}
\begin{proof}
Assume that $\mathcal X\cup\Gamma$ is endowed with the topology induced by a family $(\mathfrak F_\gamma)_{\gamma\in\Gamma}$ of proper filters.
The fact that this topology satisfies (a), (b) and (c) is straightforward. Assuming that $\mathfrak F_\gamma$ is closed under interiors, the proof of
\eqref{Fgamma} and \eqref{Fgamma2} follows by noting that, for $V\in\mathfrak F_\gamma$, we have that $\Int(V)\cup\{\gamma\}$ is open in $\mathcal X\cup\Gamma$ and thus $V\cup\{\gamma\}$ is a neighborhood of $\gamma$ in $\mathcal X\cup\Gamma$. Now assume that $\mathcal X\cup\Gamma$ is endowed with a topology
$\tau$ satisfying (a), (b) and (c) and let $\mathfrak F_\gamma$ be defined by \eqref{Fgamma}. The fact that $\mathfrak F_\gamma$ is a proper filter closed under interiors and the fact that $\tau$ is contained in the topology induced by $(\mathfrak F_\gamma)_{\gamma\in\Gamma}$ are immediate. Finally, for
$U$ open in the topology induced by $(\mathfrak F_\gamma)_{\gamma\in\Gamma}$, use the fact that $\Gamma$ is discrete in the $\tau$-subspace topology
to check that $U$ is a $\tau$-neighborhood of every $\gamma\in U\cap\Gamma$.
\end{proof}

Now we wish to study locally compact Hausdorff topologies on $\mathcal X \cup \Gamma$ induced by families of filters. We show below that such topologies can be described in terms of antichains in the quotient algebra $\cb(\mathcal X)/\rc(\mathcal X)$. Recall that the topology of a $\Psi$-space is defined through an antichain in $\wp(\omega)/\Fin(\omega)$, where $\Fin(\omega)$ denotes the ideal of $\wp(\omega)$ consisting of finite sets. The algebra $\cb(\mathcal X)/\rc(\mathcal X)$ is the appropriate generalization of $\wp(\omega)/\Fin(\omega)$ when the discrete space $\omega$ is replaced with an arbitrary locally compact Hausdorff space $\mathcal X$. For any boolean algebra $\mathcal B$ and any ideal $\mathcal J$ of $\mathcal B$, we denote by $[b]\in\mathcal B/\mathcal J$ the class of an element $b\in\mathcal B$.

\begin{prop}\label{Anticadeia}
Assume that $\mathcal X$ is locally compact and Hausdorff. Let $\mathcal X \cup \Gamma$ be endowed with the topology induced by a family of proper closed under interiors filters $(\mathfrak F_\gamma)_{\gamma \in \Gamma}$. We have that this topology is locally compact Hausdorff if and only if there exists an (automatically unique) antichain $(v_\gamma)_{\gamma \in \Gamma}$ in the quotient algebra $\cb(\mathcal X)/\rc(\mathcal X)$ such that $v_\gamma$ is a basis for the filter $\mathfrak F_\gamma$, for every $\gamma$.
\end{prop}
\begin{proof}
If $\mathcal X\cup\Gamma$ is locally compact and Hausdorff, pick a compact neighborhood of each $\gamma\in\Gamma$ of the form $V_\gamma\cup\{\gamma\}$,
with $V_\gamma$ a subset of $\mathcal X$. It is easy to see that $V_\gamma\in\cb(\mathcal X)$ and that $v_\gamma=[V_\gamma]$ is a basis for the filter
$\mathfrak F_\gamma$. The fact that $\mathcal X$ is Hausdorff implies that $(v_\gamma)_{\gamma\in\Gamma}$ is an antichain. Conversely,
let $(v_\gamma)_{\gamma\in\Gamma}$ be an antichain in $\cb(\mathcal X)/\rc(\mathcal X)$ and assume that $v_\gamma$ is a basis for $\mathfrak F_\gamma$.
Note that, if $V\in v_\gamma$, then the union of $\{\gamma\}$ with the closure of $V$ in $\mathcal X$ is a compact neighborhood
of $\gamma$ in $\mathcal X\cup\Gamma$. Finally, to see that $\mathcal X\cup\Gamma$ is Hausdorff, observe that if $C$ is a compact neighborhood of some
$x$ in $\mathcal X$ and if $V\in v_\gamma$, then $(V\setminus C)\cup\{\gamma\}$ is a neighborhood of $\gamma$ in $\mathcal X\cup\Gamma$ disjoint from $C$.
\end{proof}

We now focus on the case when $\mathcal X$ is zero-dimensional. In this case, it is possible to replace $\cb(\mathcal X)$ with a simpler boolean algebra.
More precisely, we have the following result.

\begin{prop}\label{topZerodim}
In the statement of Proposition~\ref{Anticadeia}, if $\mathcal X$ is in addition assumed to be zero-dimensional, then
we can replace $\cb(\mathcal X)/\rc(\mathcal X)$ with $\clop(\mathcal X)/\clopc(\mathcal X)$,
where $\clopc(\mathcal X)$ denotes the ideal of $\clop(\mathcal X)$ given by:
\[\clopc(\mathcal X)=\big\{C \in \clop(\mathcal X): \text{$C$ is compact}\big\}.\]
Moreover, if the topology on $\mathcal X \cup \Gamma$ is locally compact Hausdorff, then it is automatically zero-dimensional.
\end{prop}
\begin{proof}
If $\mathcal X$ is zero-dimensional and locally compact, then for every $V$ in $\cb(\mathcal X)$, there exists $C\in\clopc(\mathcal X)$ containing the
boundary of $V$ in $\mathcal X$. Thus, $V\setminus C\in\clop(\mathcal X)$ and $[V]=[V\setminus C]$. It follows that
the inclusion homomorphism of $\clop(\mathcal X)$ into $\cb(\mathcal X)$ passes to the quotient and yields an isomorphism:
\[i:\clop(\mathcal X)/\clopc(\mathcal X) \longrightarrow \cb(\mathcal X)/\rc(\mathcal X).\]
Moreover, for every $v \in \clop(\mathcal X)/\clopc(\mathcal X)$, we have that both $v$ and $i(v)$ are bases of the same filter of $\wp(\mathcal X)$.
\end{proof}

We are now interested in iterating the process described so far of defining topologies in terms of antichains by adding a new layer of points to $\mathcal X \cup \Gamma$. This involves the consideration of antichains in $\clop(\mathcal X \cup \Gamma)/\clopc(\mathcal X \cup \Gamma)$.
It is useful to have a strategy that allows us to work within a fixed boolean algebra as we continue to add layers of points. To this aim, we give the following definition.

\begin{defin}
Given a dense subset $D$ of $\mathcal X$, we set:
\begin{align*}
&\clop^D(\mathcal X)=\big\{C \cap D: C \in \clop(\mathcal X)\big\} \quad \text{and} \\
&\Clopc D(\mathcal X)=\big\{C \cap D: C \in \clopc(\mathcal X)\big\}.
\end{align*}
\end{defin}

Note that $\clop^D(\mathcal X)$ is a subalgebra of $\wp(D)$ and that:
\begin{equation}\label{eq:embD}
\clop(\mathcal X)\ni C\longmapsto C\cap D\in \clop^D(\mathcal X)
\end{equation}
is an isomorphism that carries the ideal $\clopc(\mathcal X)$ onto $\Clopc D(\mathcal X)$. Thus \eqref{eq:embD} induces an isomorphism:
\begin{equation}\label{isoD}
\clop(\mathcal X)/\clopc(\mathcal X) \longrightarrow \clop^D(\mathcal X)/\Clopc D(\mathcal X).
\end{equation}

\begin{defin}
Assume that $\mathcal X$ is locally compact Hausdorff and zero-dimensional and let $D$ be a dense subset of $\mathcal X$. Given an antichain $(v_\gamma)_{\gamma \in \Gamma}$ in $\clop^D(\mathcal X)/\Clopc D(\mathcal X)$, by the topology {\em induced\/} on $\mathcal X \cup \Gamma$ by this antichain we mean the topology induced by the family of filters $(\mathfrak F_\gamma)_{\gamma \in \Gamma}$, where $\mathfrak F_\gamma$ is the filter whose basis is the image of $v_\gamma$ under the inverse of isomorphism \eqref{isoD}.
\end{defin}

In order to describe $\clop^D(\mathcal X \cup \Gamma)$ and $\Clopc D(\mathcal X \cup \Gamma)$ in terms of the antichain $(v_\gamma)_{\gamma \in \Gamma}$, it is useful to introduce some terminology in the framework of abstract boolean algebras.

\begin{defin}
Let $\mathcal B$ be a boolean algebra and $(b_\gamma)_{\gamma \in \Gamma}$ be an antichain in $\mathcal B$. An element $b$ of $\mathcal B$ is called a {\em separator\/} of $(b_\gamma)_{\gamma \in \Gamma}$ if for each $\gamma \in \Gamma$ we have either $b_\gamma \le b$ or $b_\gamma \wedge b=0$. By a {\em trivial separator\/} of $(b_\gamma)_{\gamma \in \Gamma}$ we mean an element of $\mathcal B$ which is the supremum of a finite subset of $\big\{b_\gamma:\gamma \in \Gamma\big\}$.
\end{defin}

The separators of an antichain in $\mathcal B$ form a subalgebra of $\mathcal B$ and the trivial separators form an ideal of this subalgebra.
For simplicity, given an ideal $\mathcal J$ of $\mathcal B$, we say that a family $(b_\gamma)_{\gamma \in \Gamma}$ of elements of $\mathcal B$ is an {\em antichain modulo\/} $\mathcal J$ if $([b_\gamma])_{\gamma \in \Gamma}$ is an antichain in the quotient algebra $\mathcal B/\mathcal J$. Similarly, we say that an element $b$ of $\mathcal B$ is a {\em separator modulo $\mathcal J$\/} (resp., {\em trivial separator modulo $\mathcal J$}) of $(b_\gamma)_{\gamma \in \Gamma}$ if $[b]$ is a separator (resp., trivial separator) of $([b_\gamma])_{\gamma \in \Gamma}$.

\begin{prop}\label{clopd}
Assume that $\mathcal X$ is locally compact Hausdorff and zero-dimensional. Let $D$ be a dense subset
of $\mathcal X$ and consider $\mathcal X \cup \Gamma$ endowed with the topology induced by an antichain $(v_\gamma)_{\gamma\in\Gamma}$ in $\clop^D(\mathcal X)/\Clopc D(\mathcal X)$. Then $\clop^D(\mathcal X\cup\Gamma)$ is the subalgebra of $\clop^D(\mathcal X)$ given by:
\[\clop^D(\mathcal X\cup\Gamma)=\big\{C\in\clop^D(\mathcal X):\text{$[C]$ is a separator of $(v_\gamma)_{\gamma\in\Gamma}$}\big\}\]
and $\Clopc D(\mathcal X\cup\Gamma)$ is the ideal of $\clop^D(\mathcal X\cup\Gamma)$ given by:
\[\Clopc D(\mathcal X\cup\Gamma)=\big\{C\in\clop^D(\mathcal X):\text{$[C]$ is a trivial separator of $(v_\gamma)_{\gamma\in\Gamma}$}\big\}.\]
\end{prop}
\begin{proof}
Let $(\tilde{v}_\gamma)_{\gamma \in \Gamma}$ be the antichain in $\clop(\mathcal X)/\clopc(\mathcal X)$ corresponding to $(v_\gamma)_{\gamma \in \Gamma}$ through \eqref{isoD}. Note that a subset $C$ of $\mathcal X\cup\Gamma$ is clopen if and only if $C \cap \mathcal X \in \clop(\mathcal X)$,
$[C \cap \mathcal X] \ge \tilde{v}_\gamma$, for all $\gamma \in C \cap \Gamma$, and $[C \cap \mathcal X] \wedge \tilde{v}_\gamma=0$, for all $\gamma \in \Gamma \setminus C$.
Moreover, $C$ is a compact clopen of $\mathcal X\cup\Gamma$ if and only if either $C\in\clopc(\mathcal X)$ or $C$ is of the form
$F \cup \bigcup_{\gamma \in F} V_\gamma$, with $F\subset\Gamma$ finite and $V_\gamma\in \tilde{v}_\gamma$, for $\gamma\in F$.
\end{proof}

\subsection{The clopen algebra of a finite height space}\label{sub:finiteheight}
Let $K$ be an infinite compact Hausdorff space of finite height and denote by $K^{(n)}$ its $n$-th Cantor--Bendixson derivative. Let $N\ge0$ be such that
$K^{(N+1)}$ is finite and nonempty and set:
\[\Gamma_n=K^{(n)}\setminus K^{(n+1)},\quad n=0,1,\ldots,N+1,\]
so that $\big\{\Gamma_n:0\le n\le N+1\big\}$ is a partition of $K$. For each $n$, we have that the union:
\[\mathcal X_n=\Gamma_0\cup\Gamma_1\cup\ldots\cup\Gamma_n=K\setminus K^{(n+1)}\]
is open in $K$ and thus locally compact Hausdorff zero-dimensional. Moreover, the subspace topology of each $\Gamma_n$ is discrete and
$\mathcal X_0=\Gamma_0$ is dense in $\mathcal X_n$. It follows from Proposition~\ref{topZerodim} that the topology of $\mathcal X_{n+1}=\mathcal X_{n} \cup \Gamma_{n+1}$ is induced by
an antichain $(v^{n+1}_\gamma)_{\gamma\in\Gamma_{n+1}}$ in $\mathcal A_n/\mathcal J_n$, where:
\[\mathcal A_n= \clop^{\Gamma_0}(\mathcal X_n) \quad \text{and} \quad \mathcal J_n=\Clopc{\Gamma_0}(\mathcal X_n).\]
For each $n \ge 1$ and $\gamma \in \Gamma_n$, we pick $V^n_\gamma \in v^n_\gamma$, so that $(V^n_\gamma)_{\gamma \in \Gamma_n}$ is an antichain in $\mathcal A_{n-1}$ modulo $\mathcal J_{n-1}$. It follows from Proposition~\ref{clopd} that:
\begin{align*}
&\mathcal A_n=\big\{C \in \mathcal A_{n-1}: \text{$C$ is a separator of $(V^n_\gamma)_{\gamma \in \Gamma_n}$ modulo $\mathcal J_{n-1}$}\big\}\quad\text{and}\\
&\mathcal J_n=\big\{C \in \mathcal A_{n-1}: \text{$C$ is a trivial separator of $(V^n_\gamma)_{\gamma \in \Gamma_n}$ modulo $\mathcal J_{n-1}$}\big\}.
\end{align*}
We also set:
\[\mathcal A_{-1}=\wp(\Gamma_0),\qquad \mathcal J_{-1}=\{\emptyset\}\qquad\text{and}\qquad V^0_\gamma=\{\gamma\}, \quad \text{for $\gamma \in \Gamma_0$}.\]
Note that:
\[\mathcal J_{-1} \subset\mathcal J_0\subset\cdots\subset\mathcal J_{N+1}=\clop^{\Gamma_0}(K)=\mathcal A_{N+1}\subset\cdots\subset\mathcal A_0=\mathcal A_{-1}.\]
In the proof of our main result in Section~\ref{sec:mainresults}, it will be useful to assume that:
\begin{equation}\label{eq:partitionGamma0}
\Gamma_0=\!\!\!\bigcup_{\gamma\in\Gamma_{N+1}}\!\!\!V^{N+1}_\gamma,
\end{equation}
which is possible since $\Gamma_0\in\mathcal J_{N+1}$ and thus $[\Gamma_0]=\bigvee_{\gamma \in \Gamma_{N+1}}v^{N+1}_\gamma$.

\medskip

We conclude this section with a concrete description of the boolean algebra $\clop^{\Gamma_0}(K)$, which is isomorphic to $\clop(K)$.

\begin{prop}
In the context described above, we have that $\mathcal J_n$ is the ideal of $\clop^{\Gamma_0}(K)$ generated by:
\begin{equation}\label{eq:geraIn}
\big\{V^i_\gamma:\gamma\in\Gamma_i,\ i=0,1,\ldots,n\big\},
\end{equation}
for all $n$. Moreover, $\clop^{\Gamma_0}(K)$ is the subalgebra of $\wp(\Gamma_0)$ generated
by:
\begin{equation}\label{geraB}
\big\{V^n_\gamma:\gamma\in\Gamma_n,\ n=0,1,\ldots,N+1\big\}.
\end{equation}
\end{prop}
\begin{proof}
The first statement is easily proven by induction on $n$ and, since $\mathcal J_{N+1}=\clop^{\Gamma_0}(K)$, the second follows from the fact that
$\mathcal J_n$ is contained in the subalgebra of $\wp(\Gamma_0)$ generated by \eqref{eq:geraIn}, which is also proven by induction on $n$.
\end{proof}

\end{section}

\begin{section}{Extension of compatible measures}
\label{sec:pointdiagram}

We start by introducing some notation and recalling some elementary facts. Given a boolean algebra $\mathcal B$, we denote by $\atom(\mathcal B)$ the set of all atoms of $\mathcal B$ and by $M(\mathcal B)$ the Banach space of bounded finitely additive real-valued signed measures on $\mathcal B$, endowed with the total variation norm.
For any subset $S$ of $\mathcal B$, we denote by $\langle S\rangle$ the subalgebra generated by $S$. If $\mathcal B$ is finite, then the atoms of $\mathcal B$ form a finite partition of unity and there is a natural isomorphism between $\mathcal B$ and $\wp\big(\!\atom(\mathcal B)\big)$, which associates to each $b\in\mathcal B$ the set of atoms of $\mathcal B$ below $b$. Moreover, the map that carries each $\mu\in M(\mathcal B)$ to its restriction to $\atom(\mathcal B)$ is a linear isometry from $M(\mathcal B)$ onto the space $\ell_1\!\big(\!\atom(\mathcal B)\!\big)$
of maps $f:\atom(\mathcal B)\to\R$, endowed with the norm $\Vert f\Vert=\sum_{b\in\atom(\mathcal B)}\vert f(b)\vert$.

The set of all finite subalgebras of a boolean algebra, ordered by inclusion, is a lattice.
If $X$ is an arbitrary set, then the lattice of finite subalgebras of $\wp(X)$ is anti-isomorphic to
the lattice of all equivalence relations $\sim$ on $X$ for which the quotient $X/{\sim}$ is finite.
This anti-isomorphism associates to each finite boolean subalgebra of $\wp(X)$ the equivalence relation whose equivalence classes are the atoms of the subalgebra. It follows that,
given finite subalgebras $\mathcal B_1$ and $\mathcal B_2$ of $\wp(X)$ corresponding to equivalence relations $\sim_1$ and $\sim_2$, then
$\langle\mathcal B_1\cup\mathcal B_2\rangle$ corresponds to the intersection of $\sim_1$ with $\sim_2$ and $\mathcal B_1\cap\mathcal B_2$
corresponds to the transitive closure of the union of $\sim_1$ with $\sim_2$.

\begin{defin}\label{definconstante}
Let $\mathcal B_1$ and $\mathcal B_2$ be finite subalgebras of a boolean algebra $\mathcal B$ with $\mathcal B=\langle \mathcal B_1 \cup \mathcal B_2 \rangle$. We say that two measures $\mu_1 \in M(\mathcal B_1)$ and $\mu_2 \in M(\mathcal B_2)$ are {\em compatible\/} if they agree on $\mathcal B_1 \cap \mathcal B_2$. The {\em extension constant\/} $c(\mathcal B_1, \mathcal B_2)$ is defined as the infimum of the set of all $c \ge 0$ such that
for every pair of compatible measures $\mu_1 \in M(\mathcal B_1)$ and $\mu_2 \in M(\mathcal B_2)$ there exists a common extension $\mu \in M(\mathcal B)$ with:
\begin{equation}\label{constante}
\Vert \mu \Vert \le c(\Vert \mu_1\Vert+\Vert \mu_2 \Vert).
\end{equation}
\end{defin}

It is shown in \cite[Lemma~0.1]{Rao} that two compatible measures always admit a common extension. Moreover, it follows from \cite[Theorem~1.5]{Rao} that the extension constant of any pair of finite subalgebras is finite. The challenge is to find upper bounds for extension constants when we are dealing with an infinite family of finite boolean algebras.

\begin{defin}\label{definLep}
Given $r>0$, we say that a boolean algebra $\mathcal B$ has the {\em local extension property with constant $r$}, abbreviated as $\LEP(r)$, if there exists a collection $\Omega$ of finite subalgebras of $\mathcal B$ satisfying the following conditions:
\begin{itemize}
\item $\Omega$ is cofinal in the lattice of finite subalgebras of $\mathcal B$;
\item in every uncountable subset of $\Omega$ there exists a pair of distinct subalgebras $\mathcal B_1$ and $\mathcal B_2$ with $c(\mathcal B_1, \mathcal B_2) \le r$.
\end{itemize}
We say that $\mathcal B$ has $\LEP$ if it has $\LEP(r)$, for some $r>0$.
\end{defin}

\begin{rem}\label{remLEP}
The definition of $\LEP(r)$ appears in \cite{Plebanek} in a slightly different form. Firstly, in \cite{Plebanek} the authors consider only rational valued measures. This difference is not relevant since, arguing as in the proof of \cite[Lemma~A.1~(c)]{Plebanek}, one can show that if $\mu_1\in M(\mathcal B_1)$
and $\mu_2\in M(\mathcal B_2)$ are rational valued and if $c$ is rational, then the existence of a common extension $\mu\in M(\mathcal B)$ satisfying
\eqref{constante} is equivalent to the existence of a rational valued common extension $\mu$ satisfying \eqref{constante}.
Moreover, for rational $r$, the definition of $\LEP(r)$ given in \cite{Plebanek} becomes equivalent to ours by replacing $\Vert \mu_1 \Vert+\Vert \mu_2 \Vert$ with $\max \{\Vert \mu_1 \Vert, \Vert \mu_2 \Vert\}$ in \eqref{constante}. Therefore, if $\mathcal B$ has $\LEP(r)$ according to Definition~\ref{definLep}
and $r$ is rational, then it has $\LEP(2r)$ in the sense of \cite{Plebanek}.
\end{rem}

In what follows we introduce a visual tool that captures all the relevant information about how two finite boolean algebras are embedded in a larger one. This tool is useful for estimating extension constants.

\begin{defin}
A {\em point diagram\/} is a triple $(X;F_1,F_2)$ where $F_1$ and $F_2$ are finite sets and $X$ is a subset of $F_1\times F_2$ with
$\pi_1[X]=F_1$ and $\pi_2[X]=F_2$, where $\pi_1$ and $\pi_2$ denote the projections.
\end{defin}
A point diagram $(X;F_1,F_2)$ defines embeddings of $\wp(F_1)$ and $\wp(F_2)$ into $\wp(X)$. More precisely, we have injective
homomorphisms of boolean algebras:
\[\iota_1:\wp(F_1)\longrightarrow\wp(X) \quad \text{and}\quad\iota_2:\wp(F_2)\longrightarrow\wp(X)\]
given by:
\[\iota_1(A)=\pi_1^{-1}[A]\cap X,\ A\in\wp(F_1)\quad\text{and}\quad\iota_2(A)=\pi_2^{-1}[A]\cap X,\ A\in\wp(F_2).\]
The subalgebras $\iota_1[\wp(F_1)]$ and $\iota_2[\wp(F_2)]$ correspond, respectively, to the equivalence relations $\sim_1$ and $\sim_2$ on $X$ defined by:
\begin{equation}\label{eq:defsim12}
x\sim_1y\Longleftrightarrow\pi_1(x)=\pi_1(y)\quad\text{and}\quad x\sim_2y\Longleftrightarrow\pi_2(x)=\pi_2(y),
\end{equation}
for all $x,y\in X$. It follows that the copies of $\wp(F_1)$ and $\wp(F_2)$ in $\wp(X)$ generate $\wp(X)$, i.e.:
\[\wp(X)=\big\langle\iota_1[\wp(F_1)]\cup\iota_2[\wp(F_2)]\big\rangle.\]

Conversely, any pair of embeddings of finite boolean algebras defines a point diagram as explained in next definition.
\begin{defin}
Given finite subalgebras $\mathcal B_1$ and $\mathcal B_2$ of a boolean algebra $\mathcal B$,
the {\em point diagram of $\mathcal B_1$ and $\mathcal B_2$\/} is the triple $(X;F_1,F_2)$ defined by setting $F_1=\atom(\mathcal B_1)$, $F_2=\atom(\mathcal B_2)$ and:
\[X=\big\{(a_1,a_2)\in \atom(\mathcal B_1)\times\atom(\mathcal B_2):a_1\wedge a_2\ne0\big\}.\]
\end{defin}
Note that, if $\mathcal B= \langle \mathcal B_1 \cup \mathcal B_2 \rangle$, then we have a bijection:
\[X\ni(a_1,a_2)\longmapsto a_1\wedge a_2\in\atom(\mathcal B)\]
that induces an isomorphism between $\wp(X)$ and $\mathcal B$. Moreover, the triples $(\mathcal B,\mathcal B_1,\mathcal B_2)$ and
$\big(\wp(X),\wp(F_1),\wp(F_2)\big)$ are isomorphic in the sense that the diagrams:
\[\xymatrix@C+10pt@R+5pt{%
\wp(X)\ar[r]^-\cong&\mathcal B\\
\wp(F_1)\vbox to 12pt{\vfil}\ar@{^{(}->}[u]^{\iota_1}\ar[r]_-\cong&\vbox to 12pt{\vfil}\mathcal B_1\ar@{^{(}->}[u]}\qquad\qquad\xymatrix@C+10pt@R+5pt{%
\wp(X)\ar[r]^-\cong&\mathcal B\\
\wp(F_2)\vbox to 12pt{\vfil}\ar@{^{(}->}[u]^{\iota_2}\ar[r]_-\cong&\vbox to 12pt{\vfil}\mathcal B_2\ar@{^{(}->}[u]}\]
are commutative. Finally, note that any point diagram $(X;F_1,F_2)$ can be naturally identified with the point diagram of
the subalgebras $\iota_1[\wp(F_1)]$ and $\iota_2[\wp(F_2)]$ of $\wp(X)$.

\medskip

Now that we have established an equivalence between embeddings of pairs of finite boolean algebras and point diagrams, we are going to translate the problem
of extending compatible measures to the framework of point diagrams.
\begin{defin}
Let $(X;F_1,F_2)$ be a point diagram. We say that a pair of maps $f_1\in\ell_1(F_1)$ and $f_2\in\ell_1(F_2)$ is {\em compatible\/} if the corresponding measures
$\mu_1$ and $\mu_2$ are compatible, i.e., if for every $A_1\subset F_1$ and every $A_2\subset F_2$ with:
\[\pi_1^{-1}[A_1]\cap X=\pi_2^{-1}[A_2]\cap X\]
we have $\sum_{a\in A_1}f_1(a)=\sum_{a\in A_2}f_2(a)$. By a {\em common extension\/} of $f_1$ and $f_2$, we mean a map $g\in\ell_1(X)$ whose corresponding measure
is a common extension of $\mu_1$ and $\mu_2$, i.e.:
\[f_1(a_1)=\!\!\!\!\!\sum_{\substack{a_2\in F_2\\(a_1,a_2)\in X}}\!\!\!\!g(a_1,a_2),\ a_1\in F_1\quad\text{and}\quad
f_2(a_2)=\!\!\!\!\!\sum_{\substack{a_1\in F_1\\(a_1,a_2)\in X}}\!\!\!\!g(a_1,a_2),\ a_2\in F_2.\]
Finally, we define the {\em extension constant\/} $c(X;F_1,F_2)$ of the point diagram $(X;F_1,F_2)$ to be the infimum of the set of all $c\ge0$
such that every compatible pair $f_1\in\ell_1(F_1)$, $f_2\in\ell_1(F_2)$ admits a common extension $g\in\ell_1(X)$ with $\Vert g\Vert\le c(\Vert f_1\Vert+\Vert f_2\Vert)$.
\end{defin}
Clearly the extension constant of the point diagram $(X;F_1,F_2)$ coincides with the extension constant of the pair of subalgebras
$\iota_1[\wp(F_1)]$ and $\iota_2[\wp(F_2)]$ of $\wp(X)$.

\begin{defin}
Given a point diagram $(X;F_1,F_2)$, we call the transitive closure of the union of the equivalence relations $\sim_1$ and $\sim_2$ defined in
\eqref{eq:defsim12} the {\em rook\/} equivalence relation on $X$.
\end{defin}
Note that the equivalence classes of the rook relation are precisely the atoms of the intersection $\iota_1[\wp(F_1)]\cap\iota[\wp(F_2)]$.
Elements of this intersection can be used to decompose a point diagram as explained below.

\begin{defin}
A {\em decomposition\/} of a point diagram $(X;F_1,F_2)$ is a triple of partitions:
\begin{equation}\label{eq:decomposicao}
F_1=F_1^1\cup\ldots\cup F_1^k,\quad F_2=F_2^1\cup\ldots\cup F_2^k,\quad X=X^1\cup\ldots\cup X^k
\end{equation}
with $X^i\subset F_1^i\times F_2^i$, for $i=1,\ldots,k$. The point diagram is called {\em decomposable\/} if it admits a nontrivial
decomposition, i.e., a decomposition with $k\ge2$ and $X^i\ne\emptyset$, $i=1,\ldots,k$. Otherwise, it is called {\em indecomposable}.
\end{defin}
Note that each $(X^i;F_1^i,F_2^i)$ is itself a point diagram and that:
\[X^i=\iota_1(F_1^i)=\iota_2(F_2^i),\]
so that $X^i$ belongs to the intersection $\iota_1[\wp(F_1)]\cap\iota_2[\wp(F_2)]$, for all $i=1,\ldots,k$. Conversely,
if $X=X^1\cup\ldots\cup X^k$ is a partition of $X$ into elements of the intersection $\iota_1[\wp(F_1)]\cap\iota_2[\wp(F_2)]$, then we obtain
a decomposition of $(X;F_1,F_2)$ into point diagrams $(X^i;F_1^i,F_2^i)$ by setting
$F_1^i=\iota_1^{-1}(X^i)$ and $F_2^i=\iota_2^{-1}(X^i)$. It follows that a point diagram $(X;F_1,F_2)$ is indecomposable if and only if the
intersection $\iota_1[\wp(F_1)]\cap\iota_2[\wp(F_2)]$ is trivial. Note also that, for an indecomposable point diagram $(X;F_1,F_2)$, a pair
of maps $f_1\in\ell_1(F_1)$, $f_2\in\ell_1(F_2)$ is compatible if and only if $\sum_{a\in F_1}f_1(a)=\sum_{a\in F_2}f_2(a)$.

\begin{lem}\label{max}
If \eqref{eq:decomposicao} is a decomposition of a point diagram $(X;F_1,F_2)$, then:
\[c(X;F_1,F_2)=\max_{1\le i\le k}c(X^i;F_1^i,F_2^i).\]
\end{lem}
\begin{proof}
Note that $f_1 \in \ell_1(F_1)$ and $f_2 \in \ell_1(F_2)$ are compatible if and only if $f_1\vert_{F_1^i}$ and $f_2\vert_{F_2^i}$ are compatible for all $i$. Moreover, $g \in \ell_1(X)$ is a common extension of $f_1$ and $f_2$ if and only if $g\vert_{X^i}$ is a common extension of $f_1\vert_{F_1^i}$ and $f_2\vert_{F_2^i}$ for all $i$.
\end{proof}

\begin{lem}\label{passo}
Let $(X;F_1,F_2)$ be a point diagram. Assume that the sets $F_1$ and $F_2$ are written as disjoint unions:
\[F_1=I_1\cup J_1\quad \text{and} \quad F_2=I_2\cup J_2\]
such that the following conditions are satisfied:
\begin{itemize}
\item[(i)] $\big(X\cap(J_1\times J_2);J_1,J_2\big)$ is an indecomposable point diagram;
\item[(ii)] for all $a_1 \in I_1$, there exists $a_2 \in J_2$ with $(a_1,a_2) \in X$;
\item[(iii)] for all $a_2 \in I_2$, there exists $a_1 \in J_1$ with $(a_1,a_2) \in X$.
\end{itemize}
Then $(X;F_1,F_2)$ is indecomposable and:
\begin{equation}\label{eq:estimac}
c(X;F_1,F_2)\le c\big(X\cap(J_1\times J_2);J_1,J_2\big)+1.
\end{equation}
\end{lem}
\begin{proof}
To see that $(X;F_1,F_2)$ is indecomposable, note that every element of $X$ is rook equivalent to an element of $X\cap(J_1\times J_2)$, so that
the only rook equivalence class is $X$ itself. Now let us prove \eqref{eq:estimac}.
Let $f_1 \in \ell_1(F_1)$ and $f_2 \in \ell_1(F_2)$ be compatible. By conditions (ii) and (iii), there exist functions $\phi:I_1 \to J_2$ and $\psi:I_2 \to J_1$ such that:
\[\big(a_1,\phi(a_1)\big) \in X \quad \text{and} \quad \big(\psi(a_2),a_2\big) \in X,\]
for all $ a_1\in I_1$ and $a_2 \in I_2$. We define $\tilde f_1 \in \ell_1(J_1)$ and $\tilde f_2 \in \ell_1(J_2)$ by setting:
\begin{align*}
&\tilde f_1(a_1)=f_1(a_1)-\!\!\!\!\!\!\!\!\sum_{a_2 \in \psi^{-1}(a_1)}\!\!\!\!\!\!f_2(a_2), \ a_1 \in J_1 \quad \text{and}\\[3pt]
&\tilde f_2(a_2)=f_2(a_2)-\!\!\!\!\!\!\!\!\sum_{a_1 \in \phi^{-1}(a_2)}\!\!\!\!\!\!f_1(a_1), \ a_2 \in J_2.
\end{align*}
One readily checks that:
\[\sum_{a_1\in J_1}\tilde f_1(a_1)=\sum_{a_2\in J_2}\tilde f_2(a_2)\]
so that, by (i), $\tilde f_1$ and $\tilde f_2$ are compatible. Thus, there exists a common extension $\tilde g \in \ell_1\big(X\cap(J_1\times J_2)\big)$ of $\tilde f_1$ and $\tilde f_2$ such that:
\[\Vert \tilde g \Vert \le c\big(X\cap(J_1\times J_2);J_1,J_2\big)(\Vert \tilde f_1 \Vert+ \Vert \tilde f_2 \Vert).\]
A common extension $g \in \ell_1(X)$ of $f_1$ and $f_2$ is now obtained by setting $g(a_1,a_2)=f_1(a_1)$, for $a_1 \in I_1$ and $a_2=\phi(a_1)$, $g(a_1,a_2)=f_2(a_2)$, for $a_2 \in I_2$ and $a_1=\psi(a_2)$, $g(a_1,a_2)=\tilde g(a_1,a_2)$, for $(a_1,a_2) \in X \cap (J_1 \times J_2)$ and $g(a_1,a_2)=0$, otherwise. A straightforward computation shows that:
\[\Vert g \Vert \le \big[c\big(X\cap(J_1\times J_2);J_1,J_2\big)+1\big] (\Vert f_1 \Vert + \Vert f_2 \Vert).\qedhere\]
\end{proof}

\begin{defin}\label{type}
Given a nonnegative integer $k$, we say that a point diagram $(X; F_1,F_2)$ is {\em of type $k$\/} if the sets $F_1$ and $F_2$ can be written
as disjoint unions:
\[F_1=F_1(0)\cup\ldots\cup F_1(k)\quad\text{and}\quad F_2=F_2(0)\cup\ldots\cup F_2(k)\]
such that the following conditions hold:
\begin{itemize}
\item[(a)] $F_1(k)$ and $F_2(k)$ are singletons;
\item[(b)] $F_1(k) \times F_2(k) \subset X$;
\item[(c)] for every $i=0,\ldots,k-1$ and every $a_1 \in F_1(i)$, there exists $a_2$ in $\bigcup_{j=i+1}^kF_2(j)$ with $(a_1,a_2) \in X$;
\item[(d)] for every $i=0,\ldots,k-1$ and every $a_2 \in F_2(i)$, there exists $a_1$ in $\bigcup_{j=i+1}^kF_1(j)$ with $(a_1,a_2) \in X$.
\end{itemize}
\end{defin}

\begin{lem}\label{papa}
If $(X;F_1,F_2)$ is a point diagram of type $k$, then it is indecomposable and $c(X;F_1,F_2) \le k+\frac12$.
\end{lem}
\begin{proof}
We proceed by induction on $k$. The case $k=0$ is trivial. If $(X;F_1,F_2)$ is a point diagram of type $k$, with $k \ge 1$, then $\big(X\cap(J_1\times J_2);J_1,J_2\big)$ is a point diagram of type $k-1$, where $J_1=\bigcup_{i=1}^k F_1(i)$ and $J_2=\bigcup_{i=1}^k F_2(i)$. To conclude the proof, set $I_1=F_1(0)$, $I_2=F_2(0)$ and apply Lemma~\ref{passo}.
\end{proof}

\end{section}

\begin{section}{Main results}
\label{sec:mainresults}

The goal of this section is to prove the following result.

\begin{teo}\label{teomainLindo}
If $K$ is a compact Hausdorff space with finite height $M$, then the boolean algebra $\clop(K)$ has $\LEP\!\big(M-\frac12\big)$.
\end{teo}

This immediately implies the following generalization of \cite[Corollary~5.3]{Plebanek}.

\begin{cor}
Let $\kappa<\mathfrak c$ and assume that $\MA(\kappa)$ holds. If $K$ is a separable compact Hausdorff space with finite height and weight less than or equal to $\kappa$, then every twisted sum of $c_0$ and $C(K)$ is trivial.
\end{cor}
\begin{proof}
It follows from Theorem~\ref{teomainLindo}, \cite[Theorem~5.1]{Plebanek} and \cite[Theorem~3.4]{Plebanek}.
\end{proof}

To prove Theorem~\ref{teomainLindo}, we go back to the setup of Subsection~\ref{sub:finiteheight} and we set $\mathcal B=\clop^{\Gamma_0}(K)$. Recall that $\mathcal B$ is isomorphic to $\clop(K)$ and thus, it is sufficient to prove that $\mathcal B$ has $\LEP\!\big(M-\frac12\big)$. Our first step is to exhibit a suitable cofinal subset of the lattice of finite subalgebras of $\mathcal B$. We need some notation and terminology.
Set:
\[\mathbb G=\big\{(G_n)_{0 \le n \le N+1}: \text{$G_n$ is a finite subset of $\Gamma_n$, $n=0,1,\ldots,N+1$}\big\}\]
and for every $\mathcal G=(G_n)_{0 \le n \le N+1}\in \mathbb G$ and every subset $S$ of $\{0,1,\ldots, N+1\}$ denote by $V^S_{\mathcal G}$ the union:
\[V^S_{\mathcal G}=\bigcup\big\{V_\gamma^n: \gamma \in G_n, \ n \in S\big\}.\]

Each sequence $\mathcal G=(G_n)_{0 \le n \le N+1}$ in $\mathbb G$ defines a finite subalgebra:
\[\mathcal B_{\mathcal G}=\big\langle\big\{V_\gamma^n: \gamma \in G_n, \ n=0,1,\ldots,N+1\big\}\big\rangle\]
of $\mathcal B$ whose set of atoms we denote by $F_{\mathcal G}$.

The cofinal subset of finite subalgebras of $\mathcal B$ will consist of those $\mathcal B_{\mathcal G}$ such that the fact that $(V_\gamma^n)_{\gamma \in G_n}$ is an antichain in $\mathcal A_{n-1}$ modulo $\mathcal J_{n-1}$ is witnessed by $\mathcal G$. More precisely, we give the following definition.

\begin{defin}
We say that $\mathcal G=(G_n)_{0 \le n \le N+1}\in \mathbb G$ is {\em admissible\/} if the following conditions hold:
\begin{itemize}
\item[(a)] $V_\gamma^n \cap V_\delta^n \subset V^{\left[0,n\right[}_{\mathcal G}$, for all $n=1,\ldots, N+1$ and all $\gamma, \delta \in G_n$ with $\gamma \ne \delta$;
\item[(b)] either $V^m_\gamma \setminus V^n_\delta$ is contained in $V^{\left[0,m\right[}_{\mathcal G}$ or $V^m_\gamma \cap V^n_\delta$ is contained in $V^{\left[0,m\right[}_{\mathcal G}$, for all $m,n=1,\ldots, N+1$ with $m<n$ and all $\gamma \in G_m$ and $\delta \in G_n$;
\item[(c)] $G_{N+1}=\Gamma_{N+1}$.
\end{itemize}
\end{defin}

\begin{prop}\label{cofinalprova}
The collection:
\begin{equation}\label{escrito}
\big\{\mathcal B_{\mathcal G}: \text{$\mathcal G \in \mathbb G$ admissible}\big\}
\end{equation}
is cofinal in the lattice of finite subalgebras of $\mathcal B$.
\end{prop}
\begin{proof}
Since the set \eqref{geraB} generates $\mathcal B$, it is sufficient to show that for every $\mathcal G=(G_n)_{0 \le n \le N+1}\in \mathbb G$, there exists an admissible $\mathcal H=(H_n)_{0 \le n \le N+1}\in \mathbb G$ such that $G_n$ is contained in $H_n$, for all $n$. The sets $H_n$ are easily constructed by recursion on $n$, starting at $n=N+1$.
\end{proof}

Our next step is to obtain an explicit description of the atoms of $\mathcal B_{\mathcal G}$.
It follows easily from \eqref{eq:partitionGamma0} that, given an admissible $\mathcal G=(G_n)_{0 \le n \le N+1}\in \mathbb G$, the sets:
\[A^{n,\gamma}_{\mathcal G}=V^n_\gamma \setminus V^{\left[0,n\right[}_{\mathcal G}, \ \gamma \in G_n, \ 0 \le n \le N+1\]
form a partition of $\Gamma_0$ into elements of $\mathcal B_{\mathcal G}$. To prove that those are precisely the atoms of $\mathcal B_{\mathcal G}$, we need an auxiliary result.

\begin{lem}\label{Cs}
Let $X$ be a set and $C_1$, $C_2$, \dots, $C_n$ be subsets of $X$ such that, for $1 \le i<j \le n$, either $C_i \cap C_j$ or $C_i \setminus C_j$ is contained in $\bigcup_{k<i}C_k$. Then the subalgebra of $\wp(X)$ generated by $\big\{C_i:i=1,2,\ldots,n\big\}$ coincides with the subalgebra generated by:
\begin{equation}\label{gga}
\big\{\textstyle{C_i \setminus \bigcup_{k<i}C_k}: i=1,2,\ldots,n\big\}.
\end{equation}
\end{lem}
\begin{proof}
The proof is done by induction on $n$. Given $n>1$ and assuming that the result holds for $n-1$, we have to check that $C_n$ belongs to the subalgebra generated by \eqref{gga}. To this aim, it suffices to show that $C_n \cap C_i$ belongs to such subalgebra for all $i<n$, which is easily done by induction on $i$.
\end{proof}

\begin{cor}\label{thm:corAGatom}
If $\mathcal G$ is admissible, then the atoms of $\mathcal B_{\mathcal G}$ are:
\[F_{\mathcal G}=\big\{A^{n,\gamma}_{\mathcal G}: \gamma \in G_n, \ 0\le n\le N+1\big\}.\]
\end{cor}
\begin{proof}
Follows from Lemma~\ref{Cs} and the fact that if a subalgebra is generated by a partition, then the elements of the partition are the atoms of this subalgebra.
\end{proof}

At this point, we need to estimate the extension constant for pairs of elements of the cofinal set \eqref{escrito}.

\begin{teo}\label{main}
For every admissible sequences $\mathcal G, \mathcal H \in \mathbb G$, we have that the extension constant $c(\mathcal B_{\mathcal G}, \mathcal B_{\mathcal H})$ is less than or equal to $N+\frac32$.
\end{teo}

We can now use Theorem~\ref{main} to conclude the proof of Theorem~\ref{teomainLindo} and then the reminder of the section will be dedicated to the proof of Theorem~\ref{main}.

\begin{proof}[Proof of Theorem~\ref{teomainLindo}]
Since $\clop(K)$ is isomorphic to $\mathcal B$ and the height $M$ of $K$ is equal to $N+2$, the result follows directly from Proposition~\ref{cofinalprova} and Theorem~\ref{main}.
\end{proof}

From now on, we consider fixed admissible sequences $\mathcal G=(G_n)_{0 \le n \le N+1}$ and $\mathcal H=(H_n)_{0 \le n \le N+1}$ in $\mathbb G$ and we set:
\[\mathcal G \cap \mathcal H=(G_n \cap H_n)_{0 \le n \le N+1}.\]
Recall from Section~\ref{sec:pointdiagram} that the pair of subalgebras $\mathcal B_{\mathcal G}$, $\mathcal B_{\mathcal H}$ defines a point diagram $(X;F_{\mathcal G},F_{\mathcal H})$.
Our strategy is to use this point diagram to estimate the extension constant $c(\mathcal B_{\mathcal G}, \mathcal B_{\mathcal H})$. To this aim, we will exhibit a partition of $\Gamma_0$ into elements of $\mathcal B_{\mathcal G} \cap \mathcal B_{\mathcal H}$ which will yield a decomposition of $(X;F_{\mathcal G},F_{\mathcal H})$.

Next lemma states that $\mathcal G \cap \mathcal H$ witness the fact that $(V^n_\gamma)_{\gamma \in G_n \cap H_n}$ is an antichain modulo $\mathcal J_{n-1}$.

\begin{lem}
For any $n=1, \ldots, N+1$ and any $\gamma, \delta \in G_n \cap H_n$ with $\gamma \ne \delta$, we have $V^n_\gamma \cap V^n_\delta \subset V^{\left[0,n\right[}_{\mathcal G \cap \mathcal H}$.
\end{lem}
\begin{proof}
Since $\mathcal G$ and $\mathcal H$ are admissible, we have:
\begin{equation}\label{minhavo}
V^n_\gamma \cap V^n_\delta \subset V^{\left[0,n\right[}_{\mathcal G}\quad\text{and}\quad V^n_\gamma \cap V^n_\delta \subset V^{\left[0,n\right[}_{\mathcal H}.
\end{equation}
To prove the lemma, we will show by induction on $m$ that:
\begin{align}
\label{fome}&V^n_\gamma \cap V^n_\delta \subset V^{\left[0,m\right[}_{\mathcal G} \cup V_{\mathcal G \cap \mathcal H}^{\left[m,n\right[}\quad\text{and}\\
\label{fome2}&V^n_\gamma \cap V^n_\delta \subset V^{\left[0,m\right[}_{\mathcal H} \cup V_{\mathcal G \cap \mathcal H}^{\left[m,n\right[},
\end{align}
for $m=0,\ldots,n$. The case $m=n$ is just \eqref{minhavo}. Fix $m<n$ and assume that:
\begin{align}
&V^n_\gamma \cap V^n_\delta \subset V^{[0,m]}_{\mathcal G} \cup V_{\mathcal G \cap \mathcal H}^{\left]m,n\right[}\quad\text{and}\\
\label{semfome}&V^n_\gamma \cap V^n_\delta \subset V^{[0,m]}_{\mathcal H} \cup V_{\mathcal G \cap \mathcal H}^{\left]m,n\right[}.
\end{align}
Let us prove \eqref{fome}. The proof of \eqref{fome2} is analogous. It is sufficient to check that:
\begin{equation}\label{salgadinho}
V^n_\gamma\cap V^n_\delta\cap (V^m_\epsilon\setminus V_{\mathcal G}^{\left[0,m\right[})\subset V_{\mathcal G\cap \mathcal H}^{\left[m,n\right[},
\end{equation}
for all $\epsilon \in G_m \setminus H_m$. Since $\mathcal G$ is admissible, we have that $A_{\mathcal G}^{m,\epsilon}=V^m_\epsilon \setminus V_{\mathcal G}^{\left[0,m\right[}$ is either contained in $V^n_\gamma$ or is disjoint from $V_\gamma^n$. Similarly, $A_{\mathcal G}^{m,\epsilon}$ is either contained in $V^n_\delta$ or is disjoint from $V_\delta^n$. If $A_{\mathcal G}^{m,\epsilon}$ is disjoint from either $V^n_\gamma$ or $V^n_\delta$, then \eqref{salgadinho} holds trivially. Otherwise, it follows from
\eqref{semfome} that:
\begin{equation}\label{lucifer}
A_{\mathcal G}^{m,\epsilon}\subset V^{[0,m]}_{\mathcal H} \cup V_{\mathcal G \cap \mathcal H}^{\left]m,n\right[}.
\end{equation}
Since $A_{\mathcal G}^{m,\epsilon}$ is an atom of $\mathcal B_{\mathcal G}$ (Corollary~\ref{thm:corAGatom}) and $V_{\mathcal G \cap \mathcal H}^{\left]m,n\right[}\in\mathcal B_{\mathcal G}$, it follows that $A_{\mathcal G}^{m,\epsilon}$ is either contained in $V_{\mathcal G \cap \mathcal H}^{\left]m,n\right[}$ or is disjoint
from $V_{\mathcal G \cap \mathcal H}^{\left]m,n\right[}$. If it is contained, then \eqref{salgadinho} follows. If $A_{\mathcal G}^{m,\epsilon}$ is disjoint from $V_{\mathcal G \cap \mathcal H}^{\left]m,n\right[}$ then, by \eqref{lucifer}, $A_{\mathcal G}^{m,\epsilon}$ is contained in $V^{[0,m]}_{\mathcal H}$. Therefore:
\begin{equation}\label{eq:contradiction}
A_{\mathcal G}^{m,\epsilon}=A_{\mathcal G}^{m,\epsilon}\cap V^{[0,m]}_{\mathcal H}\subset V^{\left[0,m\right[}_{\mathcal H}\cup(V^m_\epsilon\cap V^{\{m\}}_{\mathcal H}).
\end{equation}
Since $\epsilon\not\in H_m$, we have that $V^m_\epsilon\cap V^{\{m\}}_{\mathcal H}\in\mathcal J_{m-1}$ and since $V^{\left[0,m\right[}_{\mathcal H}$ is also in $\mathcal J_{m-1}$, we obtain that \eqref{eq:contradiction} contradicts the fact that $A_{\mathcal G}^{m,\epsilon}$ is not in $\mathcal J_{m-1}$.
\end{proof}

\begin{cor}\label{atomsinter}
The sets:
\[A_{\mathcal G \cap \mathcal H}^{n,\gamma}=V^n_\gamma \setminus V^{\left[0,n\right[}_{\mathcal G \cap \mathcal H}, \ \gamma \in G_n \cap H_n, \ 0 \le n \le N+1\]
form a partition of $\Gamma_0$ into elements of $\mathcal B_{\mathcal G} \cap \mathcal B_{\mathcal H}$.\qed
\end{cor}

The partition of $\Gamma_0$ given in Corollary~\ref{atomsinter} induces a partition of $F_{\mathcal G}$:
\[\big\{F_{\mathcal G}^{n,\gamma}: \gamma \in G_n \cap H_n, \ 0 \le n \le N+1\big\},\]
where $F_{\mathcal G}^{n,\gamma}$ is the collection of atoms of $\mathcal B_{\mathcal G}$ contained in $A_{\mathcal G \cap \mathcal H}^{n,\gamma}$. Similarly, we obtain a partition of $F_{\mathcal H}$ into sets $F_{\mathcal H}^{n,\gamma}$. Next lemma gives an explicit description of $F_{\mathcal G}^{n,\gamma}$.

\begin{lem}\label{quebraatom}
For all $n=0,1,\ldots,N+1$ and all $\gamma \in G_n \cap H_n$, we have that $F_{\mathcal G}^{n,\gamma}=\bigcup_{m=0}^n F_{\mathcal G}^{n,\gamma}(m)$, where $F_{\mathcal G}^{n,\gamma}(n)=\{A_{\mathcal G}^{n,\gamma}\}$ and, for $m<n$, the elements of $F_{\mathcal G}^{n,\gamma}(m)$ are the atoms $A_{\mathcal G}^{m,\delta}$ such that $\delta \in G_m \setminus H_m$, $A_{\mathcal G}^{m,\delta} \subset V^n_\gamma$ and $A_{\mathcal G}^{m,\delta} \cap V^{\left]m,n\right[}_{\mathcal G \cap \mathcal H}=\emptyset$.
\end{lem}
\begin{proof}
Fix $n$ and $m$ with $0\le n,m\le N+1$, $\gamma \in G_n \cap H_n$ and $\delta \in G_m$. Since both $V^n_\gamma$ and
$V^{\left[0,n\right[}_{\mathcal G \cap \mathcal H}$ belong to $\mathcal B_{\mathcal G}$, we have that the atom $A_{\mathcal G}^{m,\delta}$
of $\mathcal B_{\mathcal G}$ is contained in $A_{\mathcal G \cap \mathcal H}^{n,\gamma}=V^n_\gamma\setminus V^{\left[0,n\right[}_{\mathcal G \cap \mathcal H}$ if and only if it is contained in $V^n_\gamma$ and it is disjoint from $V^{\left[0,n\right[}_{\mathcal G \cap \mathcal H}$.
If either $m>n$ or if $m=n$ and $\delta\ne\gamma$, then $A_{\mathcal G}^{m,\delta}$ is disjoint from $V^n_\gamma$. Clearly,
for $m=n$ and $\delta=\gamma$, the set $A_{\mathcal G}^{m,\delta}$ is contained in $A_{\mathcal G \cap \mathcal H}^{n,\gamma}$.
Finally, for $m<n$, we have that $A_{\mathcal G}^{m,\delta}$ is disjoint from $V^{\left[0,n\right[}_{\mathcal G \cap \mathcal H}$ if and only
if $\delta\notin H_m$ and $A_{\mathcal G}^{m,\delta} \cap V^{\left]m,n\right[}_{\mathcal G \cap \mathcal H}=\emptyset$.
\end{proof}

Obviously, we have $F_{\mathcal H}^{n,\gamma}=\bigcup_{m=0}^n F_{\mathcal H}^{n,\gamma}(m)$, with $F_{\mathcal H}^{n,\gamma}(m)$ defined as in the statement of Lemma~\ref{quebraatom}, interchanging the roles of $\mathcal G$ and $\mathcal H$.

\medskip

The partitions of $F_{\mathcal G}$ and $F_{\mathcal H}$ obtained above yield a decomposition of the point diagram $(X; F_{\mathcal G},F_{\mathcal H})$ of $\mathcal B_{\mathcal G}$ and $\mathcal B_{\mathcal H}$ by setting:
\[X^{n,\gamma}=X \cap (F_{\mathcal G}^{n,\gamma} \times F_{\mathcal H}^{n,\gamma}), \quad \gamma \in G_n \cap H_n, \ 0 \le n \le N+1.\]

\begin{lem}\label{borboleta}
For all $n=0, \ldots, N+1$ and all $\gamma \in G_n \cap H_n$, the point diagram $(X^{n,\gamma}; F_{\mathcal G}^{n,\gamma}, F_{\mathcal H}^{n,\gamma})$ is of type $n$.
\end{lem}
\begin{proof}
Let $n$ with $0\le n \le N+1$ and $\gamma \in G_n \cap H_n$ be fixed. To prove that $(X^{n,\gamma}; F_{\mathcal G}^{n,\gamma}, F_{\mathcal H}^{n,\gamma})$ is of type $n$, we consider the disjoint unions:
\[F_{\mathcal G}^{n,\gamma}=\bigcup_{m=0}^{n}F_{\mathcal G}^{n,\gamma}(m) \quad \text{and} \quad F_{\mathcal H}^{n,\gamma}=\bigcup_{m=0}^{n}F_{\mathcal H}^{n,\gamma}(m).\]
It is clear that condition (a) of Definition~\ref{type} is satisfied. To prove condition (b), we have to check that $A_{\mathcal G}^{n,\gamma}$ intersects $A_{\mathcal H}^{n,\gamma}$. If $A_{\mathcal G}^{n,\gamma} \cap A_{\mathcal H}^{n,\gamma}=\emptyset$, then $A_{\mathcal G}^{n,\gamma}$ is contained in:
\begin{equation}\label{uniaoGrande}
\bigcup \big\{A:\textstyle{ A \in \bigcup_{m<n}F_{\mathcal H}^{n,\gamma}(m)}\big\},
\end{equation}
which yields a contradiction since the union \eqref{uniaoGrande} belongs to the ideal $\mathcal J_{n-1}$ and $A_{\mathcal G}^{n,\gamma}$ does not. Now, let us prove condition (c). Fix $m<n$ and an element $A_{\mathcal G}^{m,\delta}$ of $F_{\mathcal G}^{n,\gamma}(m)$. We need to show that $A_{\mathcal G}^{m,\delta}$ intersects the union:
\[\bigcup \big\{A:\textstyle{ A \in \bigcup_{i>m}F_{\mathcal H}^{n,\gamma}(i)}\big\}.\]
Assuming that this does not hold, we have that $A_{\mathcal G}^{m,\delta}$ is contained in:
\[\bigcup \big\{A:\textstyle{ A \in \bigcup_{i \le m}F_{\mathcal H}^{n,\gamma}(i)}\big\}.\]
This implies that $A_{\mathcal G}^{m,\delta}$ is contained in the union of:
\begin{equation}\label{inferno1}
\bigcup \big\{A:\textstyle{ A \in \bigcup_{i< m}F_{\mathcal H}^{n,\gamma}(i)}\big\}
\end{equation}
with
\begin{equation}\label{inferno2}
\bigcup\big\{A_{\mathcal G}^{m,\delta} \cap A_{\mathcal H}^{m,\epsilon}: \epsilon \in H_m \setminus G_m\big\}.
\end{equation}
Clearly \eqref{inferno1} belongs to $\mathcal J_{m-1}$. Moreover, since $\delta$ is in $G_m \setminus H_m$ and thus $\delta\ne\epsilon$ for all $\epsilon\in H_m\setminus G_m$, we have that also \eqref{inferno2} belongs to $\mathcal J_{m-1}$. The fact that $A_{\mathcal G}^{m,\delta}$ is contained
in the union of \eqref{inferno1} with \eqref{inferno2} therefore contradicts the fact that $A_{\mathcal G}^{m,\delta}$ is not in $\mathcal J_{m-1}$
and concludes the proof of condition (c). The proof of condition (d) is analogous.
\end{proof}

\begin{proof}[Proof of Theorem~\ref{main}]
Using Lemmas~\ref{borboleta} and \ref{papa} we obtain that:
\[c(X^{n,\gamma}; F_{\mathcal G}^{n,\gamma}, F_{\mathcal H}^{n,\gamma}) \le n+\frac12.\]
The conclusion follows from Lemma~\ref{max}.
\end{proof}

\begin{rem}
In Theorem~\ref{teomainLindo}, if $K$ is in addition assumed to be separable, then we can improve the constant for the local extension property of $\clop(K)$.
Namely, we can obtain that $\clop(K)$ has $\LEP\big(M-\frac32\big)$. To see this, note that if $K$ is separable, then $\Gamma_0$ is countable and thus every uncountable subset of \eqref{escrito} contains a pair of distinct algebras $\mathcal B_{\mathcal G}$ and $\mathcal B_{\mathcal H}$ with $G_0=H_0$. Then it follows from Lemma~\ref{quebraatom} that $F_{\mathcal G}^{n,\gamma}(0)$ and $F_{\mathcal H}^{n,\gamma}(0)$ are empty for $n>0$ and hence the proof of Lemma~\ref{borboleta} yields that the point diagram $(X^{n,\gamma}; F_{\mathcal G}^{n,\gamma}, F_{\mathcal H}^{n,\gamma})$ has type $n-1$.
\end{rem}

\end{section}

\end{document}